\documentclass[12pt]{article}
\usepackage[a4paper,margin=1.0in]{geometry}

\usepackage[colorlinks,citecolor=magenta,linkcolor=black]{hyperref}
\pdfpagewidth=\paperwidth \pdfpageheight=\paperheight
\usepackage{amsfonts,amssymb,amsthm,amsmath,eucal,tabu,url}
\usepackage{pgf}
 \usepackage{array}
 \usepackage{tikz-cd}
 \usepackage{pstricks}
 \usepackage{pstricks-add}
 \usepackage{pgf,tikz}
 \usetikzlibrary{automata}
 \usetikzlibrary{arrows}
 \usepackage{indentfirst}
\usepackage{tabularx} 

\theoremstyle{plain}
\newtheorem{thm}{Theorem}[section]
\newtheorem{theorem}[thm]{Theorem}

\newtheorem{proposition}[thm]{Proposition}

\theoremstyle{definition}
\newtheorem{definition}[thm]{Definition}
\newtheorem{remark}[thm]{Remark}
\newtheorem{example}[thm]{Example}

\newtheorem{thevarthm}[thm]{\varthmname}

\newenvironment{varthm*}[1]{\trivlist\item[]{\bf #1.}\it}{\endtrivlist}

\renewcommand\geq{\geqslant}

\renewcommand\leq{\leqslant}

\newcommand\be{\begin{eqnarray*}}
\newcommand\ee{\end{eqnarray*}}

\newcommand\newop[2]{\def#1{\mathop{\rm #2}\nolimits}}
\newop\log{log}
\newop\ord{ord}
\newop\Gal{Gal}
\newop\SL{SL}
\newop\Bl{Bl}
\newop\mult{mult}
\newop\mass{mass}
\newop\div{div}
\newop\codim{codim}
\newop\sing{sing}
\newop\vdim{vdim}
\newop\edim{edim}
\newop\Ass{Ass}
\newop\size{size}
\newop\reg{reg}
\newop\satdeg{satdeg}
\newop\supp{supp}
\newop\Neg{Neg}
\newop\Nef{Nef}
\newop\Nefh{Nef_H}
\newop\Eff{Eff}
\newop\Zar{Zar}
\newop\MB{MB}
\newop\MBxC{MB\mathit{(x,C)}}
\newop\NnB{NnB}
\newop\Bigg{Big}
\newop\Effbar{\overline{\Eff}}

\def\keywordname{{\bfseries Keywords}}%
\def\keywords#1{\par\addvspace\medskipamount{\rightskip=0pt plus1cm
\def\and{\ifhmode\unskip\nobreak\fi\ $\cdot$
}\noindent\keywordname\enspace\ignorespaces#1\par}}
\def\subclassname{{\bfseries Mathematics Subject Classification
(2020)}\enspace}
\def\subclass#1{\par\addvspace\medskipamount{\rightskip=0pt plus1cm
\def\and{\ifhmode\unskip\nobreak\fi\ $\cdot$
}\noindent\subclassname\ignorespaces#1\par}}

\begin{document}
\title{Arrangements of inflectional lines, hyperosculating conics, and plane cubics}
\author{Artur Bromboszcz}
\date{\today}
\maketitle

\thispagestyle{empty}
\begin{abstract}
In the present note we study some arrangements of inflectional lines, hyperosculating conics, and a nodal plane cubic that are free. Moreover, we study weak combinatorics of arrangements consisting of lines, conics, and elliptic curves providing natural constraints on them.
\subclass{14N25, 14C20, 32S25}
\end{abstract}
\section{Introduction}
Our main goal in this short note is to deliver new constructions of free curves using inflectional lines and hyperosculating conics to nodal cubic curves. Our research is strictly motivated by the recent developments, namely by a paper due to Dimca, Ilardi, Pokora and Sticlaru \cite{DIPS}, where the authors constructed new examples of arrangements consisting of smooth plane curves and inflectional lines, and by a paper due to Dimca, Ilardi, Malara and Pokora \cite{DIMP}, where the authors construct examples of free plane curve using hyperosculating conics to smooth and nodal plane cubic curve. Here we want to make a merger of these two approaches by constructing new examples of free arrangements consisting of inflectional lines and hyperosculating conics to nodal plane cubic curves. Our result, Theorem \ref{thmIO}, provides a classification of such free arrangements. Moreover, we provide constraints on the weak combinatorics of such curve arrangements under the assumption that we have some naturally selected types of singularities. Our second main result is a Hirzebruch-type inequality for such selected arrangements, see Theorem \ref{HirE}.
\section{Preliminary definitions}
Let $S= \mathbb{C}[x,y,z]$ be the graded polynomial ring in three variables $x,y,z$ with complex coefficients. Let $C: f=0$ be a reduced curve of degree $d$ in the complex projective plane $\mathbb{P}^{2}_{\mathbb{C}}$. We denote by $J_f= \langle f_x,f_y,f_z \rangle$ the Jacobian ideal, so the ideal in $S$ generated by the partial derivatives.

Consider the graded $S$-module of Jacobian syzygies of $f$, namely $$AR(f)=\{(a,b,c)\in S^3 : af_x+bf_y+cf_z=0\}.$$
We define the minimal degree of non-trivial Jacobian syzygies for $f$ as
$${\rm mdr}(f):=\min\{k : AR(f)_k\neq (0)\}.$$
Now we are going to define the freeness of a reduced plane curve $C \, : f=0$ in the language of $r:={\rm mdr}(f)$ of $C$, and the total Tjurina number of $C$ which will be denoted by $\tau(C).$
\begin{definition} 
A reduced plane curve $C \, : f=0$ is free if and only if $r \leq (d-1)/2$ and
\begin{equation}\label{F1}  
 \tau(C)=(d-1)^2-r(d-r-1).
\end{equation}
\end{definition}
For the completeness of the note, we need to recall definitions of inflectional lines and hyperosculating conics. The Hessian of $f$ is defined as 
\begin{equation}
\label{eq2}
H=H(C)=\det \left(
  \begin{array}{ccccccc}
     f_{xx} & f_{xy}& f_{xz}  \\
     f_{xy} & f_{yy}& f_{yz}\\
     f_{xz} & f_{yz}& f_{zz}  \\
   \end{array}
\right).
\end{equation}

It is known that the intersection $X_C=C \cap H(C)$ consists exactly of the set of inflection points $I_C$ of $C$ union with the set of singular points $Y_C$ of $C$. Recall that
if $p \in C$ is a smooth point of this curve, and $T_pC$ denotes the projective line tangent to $C$ at $p$, then the
{\it inflection order} of $p$ is by definition
\begin{equation}
\label{ip}
\iota_p(C)=(C,T_pC)_p-2,
\end{equation}
where $(C,T_pC)_p$ denotes the intersection multiplicity of the curves
$C$ and $T_pC$ at the common point $p$. Moreover, we say that $p$ is an inflection point of $C$, i.e.,  $p \in I_C$, if and only if $\iota_p(C) >0$, and the associated tangent line at $p$ is called as \textit{the inflectional line}. 

In 1859, Cayley proved in \cite{Cayley} that for a smooth point on a plane curve of degree $d \geq 3$ there exists a unique osculating conic. It has local intersection multiplicity with the curve at the point of contact at least $5$. Then he observed that a plane curve, as for the flexes, has points where the osculating conics has contact $\geq 6$. He called these points sextactic points and the associated conics are called \textit{hyperosculating conics}. It is well-known that if $E$ is an elliptic curve, then $E$ has exactly $27$ sextactic points (so exactly $27$ hyperosculating conics), and if $C$ is a nodal cubic curve, then it has exactly $3$ hyperosculating conics, and these numbers follow form the well-known formula due to Coolidge \cite[Chapter VI, Theorem 17]{Coolidge}.

\section{Free arrangements of lines, conics and cubics}

Here is our classification result devoted to arrangements of free curves.
\begin{theorem}
\label{thmIO} 

1) If $C$ is an arrangement consisting of a nodal cubic curve, one hyperosculating conic, and one inflectional line, then the arrangement is free with the exponents $(2,3)$.

2) If $C$ is an arrangement consisting of a nodal cubic curve, one hyperosculating conic, and two distinct inflectional lines, then the arrangement is free with the exponents $(3,3)$.

\end{theorem}
\begin{proof}
Recall that all nodal cubics are projectively equivalent and we can take
$$E : \quad x^3 + y^{3} - xyz = 0$$
as our main object of studies.
Following the lines in \cite[Appendix]{S}, we can determine the three sextactic points of $E$, namely
$$P_{1} = (1:1:2), \quad P_{2} = ( \omega: \omega^2 :2), \quad P_{3} = (\omega^2: \omega:2),$$
where $\omega = e^{2\pi \iota/3}$, and the corresponding hyperosculating conics, namely
\begin{itemize}
    \item $C_{1} \, : \, 21(x^2+y^2 )-22xy-6(x+y)z+z^2 =0,$
    \item $C_{2} \, : \, 21(\omega x^2 + \omega^2 y^2 )-22xy-6(\omega^2 x+\omega y)z+z^2 =0,$
    \item $C_{3} \, : \, 21(\omega^2 x^2 +\omega y^2)-22xy-6(\omega x+ \omega^2 y)z+z^2 =0.$
\end{itemize}
Recall that at each point $P_{i}$, where the curves $E$ and $C_{i}$ meet, we get $A_{11}$-singularity.
Furthermore, as it was pointed out in \cite[Appendix]{S}, we have three flexes, namely
$$Q_{1} = (1:-1:0), \quad Q_{2} = ( 1: - \omega :0), \quad Q_{3} = (1: \omega + 1:0),$$
and the corresponding inflectional lines are the following:
\begin{itemize}
    \item $\ell_{1} \, : \, 3x + 3y + z = 0,$
    \item $\ell_{2} \, : \, 3x + 3\omega^2 y + \omega z = 0,$
    \item $\ell_{3} \, : \, 3x + 3\omega y - (\omega +1)z = 0.$
\end{itemize}
Observe that each point $Q_{j}$, where the curves $E$ and $\ell_{j}$ meet, we get $A_{5}$-singularity. Having the collected data at hand, we can proceed with our classification. 

First of all, let us consider the arrangement $F_{i,j} = \{E,C_{i}, \ell_{j}\}$ with $i,j \in \{1,2,3\}$. Observe that $\tau(F_{i,j})=19$ and this follows from the fact that our arrangement has three $A_{1}$ points, one $A_{5}$, and one $A_{11}$ point as singularities. Using \verb}SINGULAR} we can check that $r:={\rm mdr}(F_{i,j}) =2$ and hence we have
$$r^{2} - r(d-1) + (d-1)^2 = 4 - 10 + 25 = 19 = \tau(C_{i}),$$
so the arrangement $C_{i}$ is free with the exponents $(d_{1},d_{2}) = (r,d-1-r) = (2,3)$ for each $i,j \in \{1,2,3\}$.

Let us consider now $C_{i,j,k} = \{E,C_{i},\ell_{j},\ell_{k}\}$ with $i,j,k \in \{1,2,3\}$ and $j\neq k$. Observe that $\tau(C_{i,j,k})=27$ and this follows from the fact that we have exactly six $A_{1}$ points, two $A_{5}$ points, and one $A_{11}$ point as singularities. Moreover, using \verb}SINGULAR} we can check that ${\rm mdr}(C_{i,j,k})=3$ and hence we have
$$r^{2} - r(d-1) + (d-1)^2 = 9 - 18 + 36 = 27 = \tau(C_{i,j,k}),$$
so the arrangement $C_{i,j,k}$ is free with the exponents $(d_{1},d_{2}) = (r,d-1-r) = (3,3)$ for $i,j,k \in \{1,2,3\}$ and each pair of indices $j \neq k$.
\end{proof}
Using the above considerations, we can find a slightly different example of a reduced free plane curve.
\begin{example}
Let us consider the arrangement $\mathcal{C} = \{E,C_{1},\ell\}$, where 
$E$ and $C_{1}$ are the curves indicated in Theorem \ref{thmIO}, and 
$$\ell : \quad x-y=0.$$
Observe that the line $\ell$ is passing through the node of $E$ and the sextactic point. Using this observation we can detect all singularities of the curve $\mathcal{C}$, namely we have one $A_{1}$ point, one ordinary triple point $D_{4}$, and one singularity of type $D_{14}$, and hence $\tau(\mathcal{C})=19$. Since ${\rm mdr}(\mathcal{C})=2$, we have construed an example of a reduced free curve.
\end{example}
\begin{example}
Let us consider the arrangement $\mathcal{C} = \{E,\ell_{1},\ell_{2},\ell_{3}\}$, where the curves are indicated in Theorem \ref{thmIO}. The curve $\mathcal{C}$ is free. Observe that we have four singularities of type $A_{1}$ and three singularities of type $A_{5}$ and hence $\tau(\mathcal{C})=19$. Since ${\rm mdr}(\mathcal{C})=2$ it allows us to conclude that $\mathcal{C}$ is indeed free.
\end{example}
\begin{remark}
Using Theorem \ref{thmIO} we can also construct some examples of \textbf{nearly-free} curves. For instance, we can take, the arrangement $\mathcal{C}_{1} = \{E, C_{1}, \ell_{1},\ell_{2},\ell_{3}\}$, or $\mathcal{C}_{2} = \{E, C_{1}, C_{2},\ell_{2}\}$ and check that $\tau(\mathcal{C}_{1}) = \tau(\mathcal{C}_{2})=36$, which allows us to conclude that the constructed arrangements are nearly-free.
\end{remark}
\section{Combinatorial constraints on arrangements of lines, conics, and cubics with certain simple singularities}
In this section we want to focus on combinatorial constraints of arrangements consisting of lines, smooth conics, and elliptic curves in the complex projective plane admitting singularities of type $A_{1}, A_{3}, D_{4}, A_{5}, A_{7}, A_{11}, D_{14}$. Our choice of singularities is motivated by arrangements constructed in the previous section in the context of the freeness. First of all, we have the following naive count.
\begin{proposition}
\label{count}
Let $\mathcal{C} = \{\ell_{1}, ..., \ell_{d}, C_{1}, ..., C_{k}, \mathcal{E}_{1}, ..., \mathcal{E}_{l}\} \subset \mathbb{P}^{2}_{\mathbb{C}}$ be an arrangement of $d\geq 1$ lines, $k\geq 1$ smooth conics, and $l\geq 1$ elliptic curves. Assume that the arrangement $\mathcal{C}$ admits $n_{2}$ singularities of type $A_{1}$, $t_{3}$ singularities of type $A_{3}$, $n_{3}$ singularities of type $D_{4}$, $t_{5}$ singularities of type $A_{5}$, $t_{7}$ singularities of type $A_{7}$, $t_{11}$ singularities of type $A_{11}$, and $d_{14}$ singularities of type $D_{14}$. Then we have the following count
\begin{multline}
(d+2k+3l)^{2} - d - 4k - 9l = d(d-1) + 4k(k-1) + 9l(l-1) +4dk + 6dl + 12kl =  \\ 2n_{2} + 4t_{3} + 6n_{3} + 6t_{5} + 8t_{7} + 12t_{11} + 16d_{14}.
\end{multline}
\end{proposition}
\begin{proof}
It follows from B\'ezout's Theorem and from counting of the intersection indices for singular points that our curve $\mathcal{C}$ admits.
\end{proof}
Now we can present our main result in this section.
\begin{theorem}
\label{HirE}
Let $\mathcal{C} = \{\ell_{1}, ..., \ell_{d}, C_{1}, ..., C_{k}, \mathcal{E}_{1}, ..., \mathcal{E}_{l}\} \subset \mathbb{P}^{2}_{\mathbb{C}}$ be an arrangement of $d\geq 1$ lines, $k\geq 1$ smooth conics, and $l\geq 1$ elliptic curves. Assume that the arrangement $\mathcal{C}$ admits $n_{2}$ singularities of type $A_{1}$, $t_{3}$ singularities of type $A_{3}$, $n_{3}$ singularities of type $D_{4}$, $t_{5}$ singularities of type $A_{5}$, $t_{7}$ singularities of type $A_{7}$, $t_{11}$ singularities of type $A_{11}$, and $d_{14}$ singularities of type $D_{14}$. Then the following inequality holds:
\begin{equation}
27l + 8k + n_{2} + \frac{3}{4}n_{3} \geq d + \frac{5}{2}t_{3} + 5t_{5} + \frac{29}{4}t_{7} + \frac{23}{2}t_{11}+ \frac{79}{8}d_{14}.
\end{equation}
\end{theorem}
\begin{proof}
Let $C = \ell_{1} + ... + \ell_{d} + C_{1} + ... + C_{k} + \mathcal{E}_{1} + ... + \mathcal{E}_{l}$ be the divisor associated with our arrangement $\mathcal{C}$. In order to show our result we are going to use an orbifold BMY inequality by Langer \cite{Langer}, namely
$$(\star) \, : \quad \sum_{p \in {\rm Sing}(C)} 3\bigg(\alpha(\mu_{p}-1)+1-e_{orb}(p;\mathbb{P}^{2}_{\mathbb{C}},\alpha C)\bigg)\leq (3\alpha - \alpha^{2})m^{2}-3\alpha m,$$
where $e_{orb}(p;\mathbb{P}^{2}_{\mathbb{C}},\alpha C)$ is the local orbifold Euler number of a given singularity $p$, $\mu_{p}$ denotes the local Milnor number of $p$, and ${\rm deg} (C) = m$.  Using \cite[Theorem 8.7, Theorem 9.4.2]{Langer}, we obtain the following: \begin{itemize}
    \item if $p$ is singularity of type $A_{1}$, then $e_{orb}(p;\mathbb{P}^{2}_{\mathbb{C}},\alpha C) = (1-\alpha)^2$ for $\alpha \in [0, 1]$,
    \item if $p$ is singularity of type $D_{4}$, then $e_{orb}(p;\mathbb{P}^{2}_{\mathbb{C}},\alpha C) = \frac{(2-3\alpha)^2}{4}$ for $\alpha \in (0, 2/3]$,
    \item if $p$ is singularity of type $A_{3}$, then $e_{orb}(p;\mathbb{P}^{2}_{\mathbb{C}},\alpha C) = \frac{(3-4\alpha)^2}{8}$ for $\alpha \in (1/3, 3/4]$,
    \item if $p$ is singularity of type $A_{5}$, then $e_{orb}(p;\mathbb{P}^{2}_{\mathbb{C}},\alpha C) = \frac{(4-6\alpha)^2}{12}$ for $\alpha \in (1/3, 2/3]$,
    \item if $p$ is singularity of type $A_{7}$, then $e_{orb}(p;\mathbb{P}^{2}_{\mathbb{C}},\alpha C) = \frac{(5-8\alpha)^2}{16}$ for $\alpha \in (3/8, 5/8]$,
    \item if $p$ is singularity of type $A_{11}$, then $e_{orb}(p;\mathbb{P}^{2}_{\mathbb{C}},\alpha C) = \frac{(7-12\alpha)^2}{24}$ for $\alpha \in (5/12, 7/12]$,
    \item if $p$ is singularity of type $D_{14}$, then $e_{orb}(p;\mathbb{P}^{2}_{\mathbb{C}},\alpha C) = \frac{(7-13\alpha)^2}{24}$ for $\alpha \in (5/11, 7/13]$.
\end{itemize}
Since $m = d + 2k + 3l \geq 6$, we can take for our considerations $\alpha =\frac{1}{2}$, and from now one we are working with the pair $\bigg(\mathbb{P}^{2}_{\mathbb{C}}, \frac{1}{2}C\bigg)$.
We start with the left-hand side of $(\star)$, which is equal to
$$\frac{9}{4}n_{2}+\frac{45}{8}t_{3} + \frac{117}{16}n_{3} + \frac{35}{4}t_{5} + \frac{189}{4}t_{7} + \frac{143}{8}t_{11} + \frac{719}{32}d_{14}.$$
Now we look at the right-hand side of $(\star)$, namely
$$\frac{5}{4}(d+2k+3l)^{2} - \frac{3}{2}(d+2k+3l).$$
Using the naive combinatorial count presented in Proposition \ref{count}, we have
$$\frac{5}{4}(d+4k+9l+2n_{2}+4t_{3}+6n_{3}+6t_{5}+8t_{7}+12t_{11}+16d_{14}) - \frac{3}{2}(d+2k+3l).$$
Plugging the collected data above to $(\star)$ we obtain that
$$27l + 8k + n_{2} + \frac{3}{4}n_{3} \geq d + \frac{5}{2}t_{3} + 5t_{5} + \frac{29}{4}t_{7} + \frac{23}{2}t_{11}+ \frac{79}{8}d_{14},$$
which completes the proof.
\end{proof}

\section*{Conflict of Interests}
I declare that there is no conflict of interest regarding the publication of this paper.

\section*{Acknowledgement}
This note is a part of the author's Master Thesis written under supervision of Piotr Pokora.

\vskip 0.5 cm

\bigskip
Artur Bromboszcz,
Department of Mathematics,
University of the National Education Commission Krakow,
Podchor\c a\.zych 2,
PL-30-084 Krak\'ow, Poland.\\
Email Address: \verb}artur.bromboszcz@student.up.krakow.pl}
\bigskip

\begin{thebibliography}{000}
\bibitem{Cayley}
A. Cayley, On the Conic of Five-Pointic Contact at Any Point of a Plane Curve.
\textit{Philosophical Transactions of the Royal Society of London} \textbf{149}: 371 -- 400 (1859).

\bibitem{Coolidge}
J. L. Coolidge, \textit{A Treatise on Algebraic Plane Curves}. Oxford, Clarendon Press, 1931.

\bibitem{Singular}
W.~Decker, G.-M. Greuel, G.~Pfister, and H.~Sch\"onemann,
\newblock {\sc Singular} {4-1-1} --- {A} computer algebra system for polynomial computations. \newblock \url{http://www.singular.uni-kl.de}, 2018.

\bibitem{DIMP}
 A. Dimca, G. Ilardi, G. Malara, P. Pokora, Construction of free curves by adding osculating conics to a given cubic curve. \textbf{arXiv:2311.08913}.
 
\bibitem{DIPS}  A. Dimca, G. Ilardi, P. Pokora,  G. Sticlaru, Construction of free curves by adding lines to a given curve. \textit{Results Math.} \textbf{79}: Art. Id 11 - 31 pages (2024). 

\bibitem{Langer}
A. Langer, Logarithmic orbifold Euler numbers of surfaces with applications. \textit{Proc. London Math. Soc.} \textbf{86}: 358 -- 396 (2003).

\bibitem{S} K. Tono,  
On Orevkov's rational cuspidal plane curves. \textit{J. Math. Soc. Japan} \textbf{64}(2): 365--385 (2002).
\end{thebibliography}
\end{document}